\documentclass[12 pt]{article}
\usepackage{amssymb,amsmath,graphicx, amsthm,verbatim, bbold, bm, hyperref}
\newtheorem{lemma}{Lemma}
\newtheorem{theorem}{Theorem}
\newtheorem{definition}{Definition}
\newtheorem{proposition}{Proposition}
\newcommand{\<}{\langle}
\renewcommand{\>}{\rangle}
\newcommand{\M}{\mathcal{M}}
\newcommand{\ol}{\overline}
\newcommand{\beq}{\begin{equation*}}
\newcommand{\eeq}{\end{equation*}}
\numberwithin{equation}{section}
\numberwithin{theorem}{section}
\numberwithin{definition}{section}
\numberwithin{lemma}{section}
\numberwithin{proposition}{section}

\title{Hydrodynamic limit of a boundary-driven elastic exclusion process and a Stefan problem}

\author{Joel Barnes\\
\small{University of Washington}\\
\small \texttt {joel@math.washington.edu}
}
\begin{document}
\maketitle
\begin{abstract}
Burdzy, Pal, and Swanson \cite{Burdzy10} considered solid spheres of small radius moving in the unit interval, reflecting instantaneously from each other and at $x=0$, and killed at $x=1$, with mass being added to the system from the left at rate $a$. By transforming to a system with zero-width particles moving as independent Brownian motion, they derived a limiting stationary distribution for a particular initial distribution, as the width of a particle decreases to zero and the number of particles increases to infinity. This space-removing transformation has a direct analogy in the isomorphism between a new unbounded-range exclusion process and a superimposition of random walks with random boundary. We derive the hydrodynamic limit for these isomorphic processes, demonstrating that this elastic exclusion is an appropriate model for the reflecting Brownian spheres in one dimension.  \end{abstract}
\setcounter{section}{-1}
\section{Preliminaries}
This paper is primarily concerned with the hydrodynamic limit of a exclusion process $Z_t$ on a bounded one-dimensional lattice of grid size $1/N$ with the following dynamics: a particle $p$ moves into an adjacent unoccupied site at rate proportional to the size of the block of occupied particles of which $p$ is a member. We think of each particle in the block as having internal energy transferred to the outermost particle elastically. In addition, the process is boundary driven: at constant rate, the leftmost block of particles (possibly empty) is shifted to the right one position, and a new particle is added to the vacant first position. Finally, particles are killed when they move to the rightmost site. 

We call the model boundary-driven elastic exclusion. We are interested in the limiting shape of the empirical distribution for all times as the grid size scales to zero and the dynamics scale appropriately, and in theorem 5.1 we prove that this hydrodynamic limit satisfies the differential equation
\begin{align}\label{eq:hl}
\partial_t z(x,t) = \partial_x\left(\frac{1}{(1-z(x,t))^2}\partial_x z(x,t)\right),
\end{align}
with appropriate boundary conditions. This particle system was chosen to approximate the system of one-dimensional crowded Brownian spheres defined in \cite{Burdzy10}, and the connection is of interest because the hydrodynamic limit of the exclusion process $Z_t$ matches the conjectured hydrodynamic limit of the Brownian process in that paper. Because of the connection, and because the method of proof in that paper also describes the key isomorphism that we use to derive the limiting equation, we briefly describe that process and the key transformation.  H. Rost \cite{Rost84} also considered reflecting Brownian intervals and derived the hydrodynamic limit above in the case of the entire real line.

Consider intervals $I_t^k = (B_t^k,B_t^k+1/N)$, such that when $B_t^k\geq0 $ and $ |B_t^k-B_t^j|>1/N$ for all $j$ such that $ j\neq k$, $B_t^k$ moves as independent Brownian motion.  Intervals $I_t^k$ reflect instantaneously and symmetrically, and are killed when $B_t^k+1/N=1$.  Finally, for $k$ greater than some $k_0$, $B_0^k = -(k-k_0)/N$ and $B_t^k = B_0^k + at$ until $B_t^k=0$, so that particles continuously enter the interval at rate $aN$.  In order to derive the limiting stationary distribution in the case $k_0=0$, the authors of \cite{Burdzy10} consider the transformations $T_t: (-\infty,1] \rightarrow [0,S_t]$ with
\begin{equation*}
  T_t(x) =  \begin{cases}
  		0 & \text{for $x\leq0$},\\
		x - \int_0^x \mathbb{1}_{\bigcup_k I_t^k}(z) dz & \text{for $0<x\leq 1$}.
		\end{cases}
\end{equation*}
$T_t$ maps $I_t^k$ to a point $C_t^k$, and simply translates unoccupied space, so the $C_t^k$ move as independent, symmetrically reflecting Brownian motions with drift $-adt$, due to the continually inserted intervals.  Furthermore, $T_t(1) = S_t$ is a random boundary that changes proportionally to the number of particles entering or leaving the system.  Since the distribution of symmetrically reflecting Brownian motions is identical to that of independent particles, we are reduced to the case of independent $dA^k_t = dW_t^k - adt$, reflecting at $0$ and killed at $S_t$.  This leads us to conjecture the following hydrodynamic limit, which is a form of the well-known Stefan melting-freezing problem:
 \begin{definition}
Given $s_0\geq0$, $v_0\in C^1([0,s_0])$ with $s_0 = 1-\int_0^{s_0}v_0(x)dx$, and $a>0$, a pair $(v, s)$ such that $s\in C^1([0,T])$, $s(0)=s_0, s>0,$ and $v\in C^2(D_T)\cap C^1(\overline{D_T})$, where $D_T = \{(x,t): 0<x<s(t), 0<t\leq T\},$ satisfying
\begin{align}\label{eq:stefan1}
	&\partial_tv(x,t) = \partial_{xx}v(x,t) + a\partial_xv(x,t) \quad &0<x<s(t), t>0, \\
	&v(x,0) = v_0 &0\leq x\leq s(0)\label{equation:5},\\
	&(\partial_xv(x,t)+av(x,t))|_{x=0} = -a &t>0,\\
	&v(s(t),t)=0 &t>0, \\
	&s(t) = 1 - \int_0^{s(t)}v(x,t)dx &t\geq0\label{eq:stefan2},
\end{align}
is called a \emph{solution to the Stefan problem} \ref{eq:stefan1}-\ref{eq:stefan2} \emph{with initial data} $(v_0,s_0)$.
\end{definition}
Condition (\ref{eq:stefan2}) is more familiar in the differential form $s'(t) = -a - \partial_xu(x,t)|_{s(t)}$. The Stefan problem has been well studied in many forms, though perhaps not this exact form. The book \cite{Meirmanov92} by Meirmanov is an excellent reference. In particular, with $a=0$, this equation is the classical one-dimensional, one-phase melting probem, where in the region $0\leq x< s(t)$, $v$ represents the temperature of water above freezing, and $x\geq s(t)$ represents a region of ice with temperature $0$.  Strong existence and uniqueness of this case is covered in Cannon's book \cite{Cannon84}. There is every reason to believe existence holds for the equation above as well, due to the natural physical model and intrisic boundedness, but we do not take it up in this paper.  We should also note that the condition $v_0 \in C^1([0,1])$ may not be necessary, but serves only to make the definitions simpler. Indeed, the main theorem below holds whenever the initial condition is in $L^2$ and the solution satisfies the integral form \ref{eq:weakstefan}. The Stefan problem has been studied in a probabilistic context as well, as a hydrodynamic limit by Chayes and Swindle \cite{Chayes96}, Gravner and Quastel \cite{Gravner00}, Landim and Valle \cite{Landim06}, and Bertini et al \cite{Bertini99}. In \cite{Chayes96}, \cite{Landim06}, and \cite{Bertini99}, the particle model is simple exclusion, with different particle types representing the liquid and solid regions. Our model is close to that of Gravner and Quastel, who use the Stefan hydrodynamic limit of a zero-range process to prove shape theorems for internal diffusion-limited aggregation, but our proofs are not similar and the application is different. 

We now describe the second discrete process, $Y_t^N$,  which is the discrete analogue to the distribution of the transformed Brownian motion process. For notational simplicity we will omit $N$ from the process, but it will always be used in the corresponding probability measure $P^N$.  Let
\begin{align*}
A_N = \left\lbrace \frac{1}{2N}, ... , \frac{2j+1}{2N}, ... ,\frac{2N-1}{2N}\right\rbrace.
\end{align*}
  Our state space for $Y_t$ is the subset $\Omega_N$ of $\mathbb{N}^{A_N}$ such that for $\eta \in \Omega_N$, $\eta_x$ counts the number of particles at site $x$ for a distribution of particles on $A_N$ with the following restriction: there must be $M$ particles, with $M< N$, and the particles may only occupy sites $x=(2j+1)/2N$ with $j< N-M$.  Let $M_t$ be the number of particles at time $t$, and define a random boundary $S_t$ with \begin{align*}
S_t= 1 - \frac{M_t}{N} + \frac{1}{2N}.
\end{align*}
At exponential random times with rate $N^2$ for each direction, particles move as independent random walks, reflecting at the leftmost site.  If a particle hits $S_t$ at time $t$, it is killed (removed from the system), and $S_t=S_{t-}+1/N$.  In addition, there is a drift effect, occurring at rate $aN$, for $a\geq 0$ constant, where every particle except those at zero shift one site towards the origin, an additional particle is added at $1/2N$, and $S_t$ shifts one site left. The only state $\eta\in \Omega_N$ for which this does not happen is $\eta_{1/2N} = N-1$, in which case there is no change (think of the generated particle being immediately killed).
 
 The sum of delta masses of weight $Y_t/N$ at each site gives a measure $\mu_{Y_t}$ or just $\mu_t$, depending on context, of mass less than one.  The object $\mu_{\cdot}$ is an element of the Skohorod space of right-continuous paths on the metric space $\mathcal{M}$ of positive measures on $[0,1]$, $D([0,1],\mathcal{M})$. Let $P^N$ be the probability measure on right-continuous paths in $\Omega_N$ that determines the process $Y_t$. Let the corresponding probability measure on $D([0,1],\mathcal{M})$ be $Q^N$. The hydrodynamic limit of the $Q^N$ is the subject of our first theorem. For technical reasons, it is more natural to state the convergence in terms of a process $X_t$, described below in detail, such that $\rho(X_t) = Y_t$, where $\rho(x) = (x-1)\lor 0$. $\Omega_N^*$ is the state space for $X_N$, in one-to-one correspondence with $\Omega_N$. A precise version of the following theorem is found at \ref{thm:2}.

\begin{theorem}\label{thm:1}
Suppose that for each $P^N$, $X^N_0$ converges weakly to a measure with fixed density $u_0\in L^2([0,1])$. Then the empirical measures of $X_t^N$ converge to the unique solution of a weak version (\ref{eq:weakstefan}) of the Stefan problem (\ref{eq:stefan1})-(\ref{eq:stefan2}) with initial data $u_0$. If a solution for the problem (\ref{eq:stefan1})-(\ref{eq:stefan2}) with initial data $(v_0,s_0) = (\rho(u_0),\inf\{x:u_0(x)=0\})$ exists, then the empirical measures of the process $Y_t$ converges to that solution in probability.
\end{theorem}

The weak version of the problem is defined in section 3. The methods used to derive the hydrodynamic limit are largely based on those in the book  \cite{Kipnis99} by Kipnis and Landim, so we identify our contribution in two areas. First, although the Stefan problem has been well studied as a hydrodynamic limit, the exclusion process we describe and the application of the free boundary problem to such a process is new. Hydrodynamics of exclusion processes is an active field of research, but the most general results are for gradient systems with finite-range interactions, as in \cite{Farfan11}. The interactions of $X_t$ have unbounded range.  Second, the simultaneous drift effect of the transformed process is unusual, but required by the isomorphism.  The nonstandard process and the simple setting in the unit interval allows for an interesting application of elementary harmonic analysis for the $H_{-1}$ bound and the uniqueness proof. 

 Our general approach, following \cite{Gravner00}, is to make the free boundary go away by building it into the zero-range dynamics of a process $X^N_t$ as described below.  From the other direction, the differential equation transforms into a nonlinear integral equation which mirrors the form of the process. The proof of Theorem \ref{thm:1} is in four steps. In Lemma \ref{lem:relativecompactness}, we prove that the Markov process describes a relatively compact sequence of probability measures. Lemma \ref{lem:l2} guarantees that any limit points lie in $L^2([0,1]\times[0,T])$ almost surely. Lemma \ref{lem:weaksolution} shows that such limit points must satisfy a weak version of equation (\ref{eq:stefan1})-(\ref{eq:stefan2}). Finally, Lemma \ref{lem:weakuniqueness} proves that the solution of such an equation is unique. Combining these results, we see that the process converges to a measure which is the delta measure on the solution of an integral form of the problem which coincides with the solution to that problem when it exists. 

In section 5, we show that the two discrete processes described in this section are in fact isomorphic, and use the isomorphism to prove the hydrodynamic limit (\ref{eq:hl}) in Theorem \ref{thm:5}. 

\section{Construction and relative compactness}

We construct a Markov process $X_t^N$, henceforth $X_t$, by defining its infinitesimal generator. For $N>0$ let $A_N = \{1/2N,3/2N,..(2N-1)/2N\} $ and $ \mathcal{M}_N = \mathbb{N}^{A_N}$.  Let $\mathcal{M}$ be the set of finite measures on $[0,1]$, and we associate $\eta\in\mathcal{M}_N$ with its empirical measure in $\mathcal{M}, \mu_\eta=\sum_{x\in A_N} \frac{\eta_x}{N}\delta_{x}$.  Consider the following generator on functions $f: \mathcal{M}_N \rightarrow \mathbb{R}$:
\begin{align*}
\mathcal{L}_Nf(\eta) =&{} \mathcal{L}^1_Nf(\eta) + \mathcal{L}^2_Nf(\eta),\\\noalign{\vskip 2mm}
\mathcal{L}^1_Nf(\eta) =&{} N^2 \sum_{x\in A_N}\lambda (\eta_x ) \left[f(\eta^{x,x-1/N})-f(\eta)+f(\eta^{x,x+1/N})-f(\eta)\right],\\
\mathcal{L}^2_Nf(\eta) =&{} aN\left(f(\sigma(\eta))-f(\eta)\right),
\end{align*}
where $\rho(x) = (x-1)\lor 0$,
\begin{equation*}
\eta^{x,x+i/N}_y  = \begin{cases}
			\eta_{y}-1 &\text{for $y=x \text{ and }x + i/N \in A_N$ },\\
			\eta_{y}+1 &\text{for $y = x + i/N$},\\
			\eta_{y} &\text{otherwise},
			\end{cases}
\end{equation*}
and
\begin{equation*}
\sigma(\eta)_{y} = \begin{cases}
	\eta_{y}+\eta_{y+1/N} &\text{for $y=1/2N$},\\
	0 &\text{for $y=(2N-1)/2N$},\\
	\eta_{y+1/N} &\text{otherwise}.
	\end{cases}
\end{equation*}

The appendix of \cite{Kipnis99} describes the construction of a Markov process $X_t$ on $\mathcal{M}_N$ from such a generator (such that $d/dtE\left[f(X_t)\mid X_s=\eta\right]|_{t=s}=\mathcal{L}f(\eta)$), the idea being that states are changed at the minimum of exponential random times with rates $N^2\rho(\eta_x)$ and $aN,$ to the corresponding state, with the minimum itself being an exponential random time, well defined almost surely. 

 Let  $\Omega'_N$ be the set of states $\eta$ such that $\sum_{x \in A_N}\eta_x = N$, $\eta_x>0$ for $x$ less than some $b$ and $\eta_x=0$ for $x \geq b$.  When $X_0 \in \Omega'_N$ a.s., the process $\rho(X_t)$ is equal in distribution to the process $Y_t^N$ described in the introduction since it has the same dynamics.  The reflecting random walk effect is due to the fact that a pile at site $x$ loses particles at a rate proportional to the height $\rho(X_t(x))$, as if each particle is moving independently in each direction at rate $N^2$.  We now consider the drift and random boundary.  Let $S_t = \min\{x \in A_N:X_t(x)=0\}$.  When a particle moves from $S_t-1/N $ to $S_t$, which can only happen if $X_t(S_t-1/N) \geq 2$, it is killed, in the sense that $\rho(X_t)$ no longer counts it, and the boundary $S_t$ is incremented by $1/N$, as desired.  With one exception, when the process moves from $\eta$ to $\sigma(\eta)$, $\rho(\sigma(\eta_x)) = \rho(\eta_{x+1})$ except at $x=0$, where $\rho(\sigma(\eta)_{0}) = \rho(\eta_{1/N})+\rho(\eta_{0}) + 1$, representing the generated particle. The only exception is the state $X_t(1/2N) = N$, $S_t=3/2N$, in which case $\sigma(\cdot)$ has no effect, again as desired. Thus $\rho(X_t)$ with boundary $S_t$ is identical to $Y_t$ in its dynamics, and therefore in distribution, given corresponding initial distributions. Finally, note that $\Omega'_N$ is closed under the process, and can function as the state space, corresponding to the state space $\Omega_N$ of $Y_t$. 

Next, we calculate the generator applied to a linear functional.  First, we have
\begin{align*}
	\mathcal{L}_N\eta_x = \Delta_N^*\rho(\eta_x) - aD_N^*\eta_x,
\end{align*}
where $\Delta_N$ and $D_N$ are operators with $N\times N$ matrices:
\begin{align*}
\Delta_N = N^2\begin{pmatrix}
	-1 & 1 & 0 & 0 & \ldots \\
	1 & -2 & 1 & 0 & \ldots \\
	\ddots & \ddots & \ddots & \ddots & \ddots \\
	\ldots & 0 & 1 & -2 & 1 \\
	\ldots & 0 & 0 & 1 & -1
	\end{pmatrix}
\end{align*} and
\begin{equation*}
D_N = N\begin{pmatrix}
	0 & 0 & 0 & 0 & \ldots \\
	-1 & 1 & 0 & 0 & \ldots \\
	0 & -1 & 1 & 0 & \ldots \\
	\ddots & \ddots & \ddots & \ddots & \ddots \\	
	\ldots & 0 & 0 & -1 & 1
	\end{pmatrix}.
\end{equation*}
Here $A^*$ denotes the transpose of $A$. Note that $\Delta_N$ and $D_N$ represent the second symmetric and first left difference quotients, respectively, for functions $f\in C^2([0,1])$ with $f'(0) = f'(1) = 0$, in that for these functions, $D_Nf (x)$ converges to $f''(x)$ as $N$ goes to infinity, uniformly in $A_N$. 

We briefly digress to discuss the choice of our class of functions. Throughout the paper, the functions $f\in C^2([0,1])$ with $f'(0)=f'(1)=0$ will be used for test functions, as they serve several purposes. First, as noted, it is these functions for which the operators above converge uniformly to $f''$ and $f'$, respectively. Second, they are as a dense class of functions in $C([0,1])$ and can be used to define a metric on $\mathcal{M}$. Third, it is against these test functions that the weak form of the Stefan problem holds. Finally, the subfamily $\{\sqrt{2}\cos(\pi k x)\}$ is an orthonormal basis for the discrete and continuous domains, and is used to prove essential $L^2$ bounds later in the paper. 

Returning to our calculation, for $f:[0,1]\rightarrow \mathbb{R}$, let 
\beq
\<f, \eta\>_N = \frac{1}{N} \sum_{x \in A_N} f(x)\eta_x,
\eeq
 and by linearity of $\mathcal{L}_N,$
\begin{align}\label{equation:generator}
\nonumber \mathcal{L}_N\<f, \eta\>_N& = \frac{1}{N}\sum_{x \in A_N} f(x) (\Delta^*_N\rho(\eta_x) - aD_N^*\eta_x)\\
	& = \<\Delta_Nf,\rho(\eta_\cdot)\>_N - a\<D_Nf,\eta\>_N.
\end{align}

Next we prove relative compactness of $X$ in $D([0,T], \mathcal{M})$. Let
\begin{align*}
\<f,\mu\> = \int f(x)d\mu(x).
\end{align*}
We define a metric on $\mathcal{M}$, the space of positive measures on $[0,1]$, letting 
\begin{equation}
d(\nu,\mu) = \sum_{j=0}^\infty \frac{|\<f_j,\nu\>-\<f_j,\mu\>|\wedge1}{2^j},
\end{equation}
where $f_j$ are in $C^2([0,1])$ with $f'(0)=f'(1)=0$, a set which is dense in $C([0,1])$. When $\mu(dx) = u(x)dx$, we will use $\<f,u\>$ and $\<f,\mu\>$ interchangeably. Since each $f_j$ is bounded, $A\in\mathcal{M}$ is precompact if and only if $\< 1, \mu\>$ is bounded over $A$ (each $\<f_j,\mu\>$ is bounded and converges along a subsequence, which implies subsequential convergence in the metric, and $\mathcal{M}$ is complete with respect to $d$).  Note that the supports of all $Q_N$ are contained in $\mathcal{M}_1=\{\mu:\< 1, \mu\>=1\}$, a compact set. Convergence is weak convergence (convergence of expectations of continuous functions) in the space of probability measures on the Skohorod space $D([0,T],\mathcal{M})$ of right-continuous functions on $\M$.  For our purposes, convergence in this space can be convergence in $\mathcal{M}$, uniformly in $t$, since our limit points are continuous. Thus if $f(\mu)$ is continuous on $\mathcal{M}$, then $\int_0^Tf(\mu_t)dt$ and $\sup_{0\leq t\leq T}f(\mu_t)$ are continuous on $D([0,T],\mathcal{M})$. By $\{X_t\}_{t=0}^T$, we will mean the Markov process on $A_N$ with probability measures $P^N$. By $\mu_t$ we will mean the corresponding coordinate process on $\mathcal{M}_1$, with probability measure $Q^N$, so that, for example, for $A$ Borel, 
\beq
P^N\left[\<f,X_t\>_N\in A\right] = Q^N\left[\<f,\mu_t\>\in A\right].
\eeq
Relative compactness in this space follows from the following conditions, found in Chapter 2 of \cite{Kipnis99}. Let $\mathcal{T}_T$ be the space of stopping times of the usual filtration, bounded by $T$. 
\begin{lemma}
Let $Q^N$ be a sequence of probability measures on $D([0,T],\M)$. The sequence is relatively compact (in the sense of weak convergence) if:
\begin{enumerate}
\item For every $t$ in $[0,T]$ and every $\epsilon>0$, there is a compact $K(t,\epsilon) \subset \M$ such that $\sup_N Q^N [\mu_t \notin K(t,\epsilon)] \leq \epsilon$.
\item
	\beq
		\lim_{\gamma\rightarrow0} \limsup _{N\rightarrow \infty} \sup_{\tau \in \mathcal{T}_T, \theta\leq\gamma} P^N[\rho(\mu_\tau,\mu_{(\tau+\theta)\land T})\geq \epsilon ] = 0.
	\eeq
\end{enumerate}
\end{lemma}
We prove the following lemma by checking these conditions.
\begin{lemma}\label{lem:relativecompactness}
For an initial distribution such that $X_0\in \Omega_N'$ a.s., the sequence $\{Q^N\}$ is relatively compact in $D([0,T],\mathcal{M})$.
\end{lemma}
\begin{proof}

Note that condition $(1)$ is automatically satisfied since, for all $N$, $P^N[\mu_t \notin \M_1]=0$. To check $(2)$, we determine the square variation process for the $P^N$-martingale $M_t = \<f,X_t\>_N-\int_0^t\mathcal{L}_N\<f,X_s\>_Nds$ for $f\in C^2([0,1])$ with $f'(0)=f'(1)=0$. Exactly as in the proof of Theorem 3.2 of \cite{Burdzy06}, we can show that $M_t^2 - \int_0^t B_s ds$ is a martingale, where, on $\{X_t = \eta\}$,
\begin{align}
	\nonumber B_t = &\lim_{s\rightarrow 0^+}(1/s)E^N[(\<f,X_{t+s}\>_N-\<f,X_t\>_N^2\mid X_t]\\
	\nonumber	=&\, N^2\sum_{x\in A_N} \rho(\eta_x)\left[(\<f,\eta^{x,x+1}\>_N-\<f,\eta\>_N)^2 + (\<f,\eta^{x,x-1}\>_N-\<f,\eta\>_N)^2\right]\\
	\nonumber	&+ aN(\<f,\sigma(\eta)\>_N-\<f,\eta\>_N)^2\\
		=& \, N^2\sum_{x\in A_N}\rho(\eta_x)\frac{1}{N^4}(D_Nf(x+1)^2 + D_Nf(x)^2)+ a\frac{1}{N}\<D_Nf,\eta\>_N^2.
\end{align}
Since $|D_Nf(x)|\leq\|f'\|_\infty$, and $1/N\sum_{x\in A_N} \rho(\eta_x)\leq 1/N\sum_{x\in A_N}\eta_x =1$,
\begin{equation}
|B_t| \leq \frac{1}{N}\|f'\|^2 + \frac{a}{N}\|f'\|^2.
\end{equation}
Fix $\tau \in \mathcal{T}_T$, and by $\tau+\theta$ we will mean $(\tau+\theta)\land T$, and
\beq
E^N\left[M_{\tau+\theta}^2 - \int_0^{\tau+\theta} B_s ds \mid \mathcal{F}_\tau\right] = M_\tau^2 - \int_0^\tau B_s ds,
\eeq
so
\beq 
E^N\left[M_{\tau+\theta}^2 - M_{\tau}^2\right] = E^N\left[\int_\tau^{\tau+\theta}B_sds\right] \leq \frac{C\theta}{N}.
\eeq
Now
\begin{align}\label{equation:martingalebound}
 \nonumber |\<f,X_{\tau+\theta}\> - \<f,X_\tau\>|&\leq |M_{\tau+\theta}-M_\tau | + \left|\int_\tau^{\tau+\theta}\mathcal{L}_N\<f,X_s\>ds\right|,\\
\nonumber P[|M_{\tau+\theta}-M_\tau| \geq \epsilon] &\leq \frac{E[ (M_{\tau+\theta}-M_\tau)^2]}{\epsilon^2}\\
	\nonumber &= \frac{E[M_{\tau+\theta}^2-M_\tau^2]}{\epsilon^2}\\
	&\leq \frac{C\theta}{N\epsilon^2},
\end{align}
and  $|\int_\tau^{\tau+\theta}\mathcal{L}_N\<f,X_s\>ds|\leq C\theta$, since the generator is the inner product of derivatives of $f$ with measures of bounded mass.  Thus, 
\beq
P^N[|\<f,X_{\tau+\theta}\> - \<f,X_\tau\>|\geq\epsilon]\leq C_\epsilon\theta.
\eeq
To bound the metric by a given $\epsilon$, we only need consider finitely many $f_k$ and choose $\epsilon_k$ appropriately for each of these. The bound is independent of $\tau$, so 
\begin{equation}
\sup_{\tau \in \mathcal{T}_T, \theta\leq\gamma} P^N[\rho(\mu_\tau,\mu_{\tau+\theta})\geq \epsilon ] \leq C_\epsilon\gamma
\end{equation}
and $(2)$ is satisfied. Thus $Q^N$ is relatively compact and has subsequential limits. 
\end{proof}

\section{Limit measures are $L^2$ almost surely}
In this section, we prove that subsequential limits of the measures $Q^N$ are uniformly bounded in $L^2$, depending on the $L^2$ norm of the limiting initial distribution. This allows us to apply the convergence and uniqueness results of later sections.
\begin{lemma}\label{lem:l2}
If for each $N$, $X_0^N$ under $P^N$ is a random variable with values a.s. in $\Omega'_N$ such that $\sup_NE^N[1/N\sum_{x\in A_N}X_0(x)^2]<\infty$, a subsequential limit $Q^\infty$ of the corresponding $Q^N$ on $D([0,T],\mathcal{M})$ has the property that $\mu$ is absolutely continuous with density $u \in L^2([0,T]\times [0,1])$, $Q^\infty$-a.s.
\end{lemma}
\begin{proof}
We prove Lemma \ref{lem:l2} by looking at the evolution of a variant of the $H_{-1}$ norm. Let $\psi_0 = 1$ and $\psi_k = \sqrt{2}\cos(k\pi x)$, and let $\phi_0 = 0$ and $\phi_k = \sqrt{2}\sin(k\pi (x-1/2N)).$  Recall that $A_N = \{1/2N, ... , (2j+1)/2N, ... (2N-1)/2N\}$ and let $\bm{\psi_k^N}, \bm{\phi_k^N}$ represent the vectors of the respective functions evaluated on $A_N$. Then $\{\psi_k\}_{k=0}^\infty$ has the following properties:
\begin{enumerate}
\item $\Delta_N \bm{\psi_k^N} = -4N^2\sin^2(\pi k /2N)\bm{\psi_k^N}$.
\item $D_N \bm{\psi_k^N} = - 2N\sin(\pi k/2N)\bm{\phi_k^N}$.
\item $\{\bm{\psi_k^N}\}_{k=0}^{N-1}$ is an orthonormal basis for $\mathbb{R}^{A_N}$.
\item $\<\eta_1,\eta_2\>_N = \sum_0^{N-1}\<\psi_k,\eta_1\>_N\<\psi_k,\eta_2\>_N$ for $\eta_1,\eta_2\in \mathbb{N}^{A_N}$.
\item $\eta = \sum_0^{N-1}\<\psi_k,\eta\>_N\bm{\psi_k^N}$ for  $\eta\in \mathbb{N}^{A_N}$.

\end{enumerate}
The first three can be easily checked by calculations, and the last two are consequences of $(3)$. Let $\lambda_{k,N} = 2N\sin(\pi k/2N)$. Next we consider, for $\eta\in\mathbb{N}^{A_N}$,
\beq
h_N(\eta) = \sum_{k=1}^{N-1}\frac{\<\psi_k,\eta\>_N^2}{\lambda_{k,N}^2}.
\eeq
We restrict to $\eta\in\Omega'_N$ so that $\<\psi_0,\eta\>_N=1$. For $1\leq k<N,$ we have $\lambda_{k,n}>C>0$, independent of $k$ and $N$, so 
\begin{align*}
h_N(\eta)\leq \sum_{k=0}^{N-1}\<\psi_k,\eta\>^2 = \frac{1}{N}\sum_{A_N}\eta_x^2.
\end{align*}
Apply the generator
\begin{align*}
\mathcal{L}_N^1\<\psi_k,\eta\>^2 =&{} \, N^2\sum_{x \in A_N} \rho(\eta_x)\left[\<\psi_k,\eta^{x,x+1}\>_N^2-\<\psi_k,\eta\>_N^2+\<\psi_k,\eta^{x,x-1}\>_N^2-\<\psi_k,\eta\>_N^2\right]\\
		=&{} -2\lambda_{k,N}^2\<\psi_k,\rho(\eta_\cdot)\>_N\<\psi_k,\eta\>_N\\
		&+\sum_{x\in A_N}\rho(\eta_x)\left[\left(\psi_k(x+\frac{1}{N})-\psi_k(x)\right)^2+\left(\psi_k(x-\frac{1}{N})-\psi_k(x)\right)^2\right]\\
		  \leq&{} -2\lambda_{k,N}^2\<\psi_k,\rho(\eta_\cdot)\>_N\<\psi_k,\eta\>_N+\lambda_{k,N}^2/N,
\end{align*}
where the last bound follows from the inequalities $N(\psi_k(x-1/N)-\psi_k(x))\leq \lambda_{k,N}$ and $\<1,\rho(\eta_\cdot)\>_N \leq 1$. Also,
\begin{align*}
\mathcal{L}_N^2\<\psi_k,\eta\>_N^2 &= -aN(\<\psi_k,\sigma(\eta)\>_N^2-\<\psi_k,\eta\>_N^2)\\
&= a(2\<\psi_k,\eta\>_N+N^{-1}\lambda_{k,N}\<\phi_k,\eta\>_N)\lambda_{k,N}\<\phi_k,\eta\>_N\\
&\leq 2a\lambda_{k,N}\<\psi_k,\eta\>_N\<\phi_k,\eta\>_N + \lambda_{k,N}^2/N,
\end{align*}
and together we get
\begin{align*}
\mathcal{L}_Nh_N(\eta) \leq{}& -2\sum_{k=1}^{N-1}\<\psi_k,\rho(\eta_\cdot)\>_N\<\psi_k,\eta\>_N\\
	&+2a\sum_{k=1}^{N-1}\frac{\<\psi_k,\eta\>_N\<\phi_k,\eta\>_N}{\lambda_{k,N}}+C\\
	=& -2\<\rho(\eta_\cdot),\eta\>^2+C_1\sum_{k=1}^{N-1}\frac{\<\psi_k,\eta\>_N\<\phi_k,\eta\>_N}{\lambda_{k,N}}+C_2.
\end{align*}
Next, let $b_k=\<\psi_k,\eta\>_N$, and consider
\begin{align*}
	\sum_{k=1}^{N-1}\frac{\<\psi_k,\eta\>_N\<\phi_k,\eta\>_N}{\lambda_{k,N}}&=\sum_{k=1}^{N-1}\frac{b_k\<\phi_k,\sum_{j=0}^{N-1}b_j\psi_j\>_N}{\lambda_{k,N}}\\
		&= \sum_{k=1}^{N-1}\sum_{j=0}^{N-1}\frac{b_kb_j\<\phi_k,\psi_j\>_N}{\lambda_{k,N}}.
\end{align*}
We claim that
\begin{multline}
	\sum_{i=0}^{M-1}\sin\left(\frac{ \pi ki}{N}\right)\cos\left(\frac{\pi j(2i+1)}{2N}\right)=\\
	\frac{1}{4}\csc\left(\frac{\pi(j+k)}{2N}\right)\left(\cos\left(\frac{2\pi j + \pi k}{2N}\right)-\cos\left(\frac{2\pi j M + 2\pi k M - \pi k}{2N}\right)\right)\\
	+\frac{1}{4}\csc\left(\frac{\pi(j-k)}{2N}\right)\left(\cos\left(\frac{2\pi j M - 2\pi k M + \pi k}{2N}\right)-\cos\left(\frac{2\pi j -\pi k}{2N}\right)\right),
\end{multline}
and we prove by induction in $M$.  Since the constant terms of the right hand side are the variable terms evaluated at $M=1$, we see that we can check the difference $S(M+1)-S(M)$ to obtain a telescoping sum on the right hand side.  In other words, we require
\begin{align*}
	&{}\sin\left(\frac{ \pi kM}{N}\right)\cos\left(\frac{\pi j(2M+1)}{2N}\right)=\frac{1}{4}\csc\left(\frac{\pi(j+k)}{2N}\right)\\
	  &\times\left(\cos\left(\frac{2\pi jM + 2\pi kM-\pi k}{2N}\right)-\cos\left(\frac{2\pi j (M+1) + 2\pi k (M+1) - \pi k}{2N}\right)\right)\\
	&+\frac{1}{4}\csc\left(\frac{\pi(j-k)}{2N}\right)\\
	&\times\left(\cos\left(\frac{2\pi j (M+1) - 2\pi k (M+1) + \pi k}{2N}\right)-\cos\left(\frac{2\pi jM -2\pi kM + \pi k}{2N}\right)\right).
\end{align*}

Recalling that $\cos(A+B)-\cos(A) = -2\sin(B/2)\sin(A+B/2)$, first with $A= (2\pi jM + 2\pi kM-\pi k)/2N$ and $B = \pi (j+k)/N$, the first term on the right hand side becomes
\begin{equation*}
	\frac{1}{2}\sin\left(\frac{2\pi jM + 2\pi kM+\pi j }{2N}\right),
\eeq
and with $A =  (2\pi jM - 2\pi kM+\pi k)/2N$ and $B = \pi (j-k)/N$, the second term becomes
\beq
	-\frac{1}{2}\sin\left(\frac{2\pi jM - 2\pi kM + \pi j }{2N}\right).
\eeq
Finally, the difference formula for $\sin$ is $\sin(A+B)-\sin(A) =2\sin(B/2)\cos(A+B/2)$, and substituting $A = (2\pi jM - 2\pi kM + \pi j)/2N$ and $B = 2\pi kM/ N$, we get the desired formula.

Therefore, for $j-k$ even, we get
\begin{align*}
	\<\phi_k,\psi_j\>_N ={}&\frac{1}{4N}\csc\left(\frac{\pi(j+k)}{2N}\right)\left(\cos\left(\frac{2\pi j + \pi k}{2N}\right)-\cos\left(\frac{- \pi k}{2N}\right)\right)\\
	&+\frac{1}{4N}\csc\left(\frac{\pi(j-k)}{2N}\right)\left(\cos\left(\frac{\pi k}{2N}\right)-\cos\left(\frac{2\pi j -\pi k}{2N}\right)\right)\\
	={}&\frac{1}{4N}\csc\left(\frac{\pi(j+k)}{2N}\right)\left(\cos\left(\frac{\pi k}{2N} + \frac{\pi j}{N}\right)-\cos\left(\frac{\pi k}{2N}\right)\right)\\
	&+\frac{1}{4N}\csc\left(\frac{\pi(j-k)}{2N}\right)\left(\cos\left(\frac{-\pi k}{2N}\right)-\cos\left( \frac{-\pi k}{2N}+\frac{\pi j}{N}\right)\right),
\end{align*}
and using the identity $\cos(A+B)-\cos(A) = -2\sin(A+B/2)\sin(B/2)$, we get
\begin{align*}
	\<\phi_k,\psi_j\>_N ={} -\frac{1}{2N}\left(\sin\left(\frac{\pi j}{2N}\right)+\sin\left(\frac{-\pi j}{2N}\right) \right) = 0.
\end{align*}
For $j-k$ odd, we get
\begin{align*}
	\<\phi_k,\psi_j\>_N ={}&\frac{1}{4N}\csc\left(\frac{\pi(j+k)}{2N}\right)\left(\cos\left(\frac{2\pi j + \pi k}{2N}\right)+\cos\left(\frac{- \pi k}{2N}\right)\right)\\
	&+\frac{1}{4N}\csc\left(\frac{\pi(j-k)}{2N}\right)\left(-\cos\left(\frac{\pi k}{2N}\right)-\cos\left(\frac{2\pi j -\pi k}{2N}\right)\right),
\end{align*}
and this time, with $\cos(A+B)+\cos(A) = 2\cos(A+B/2)\cos(B/2)$,
\begin{align*}
	\<\phi_k,\psi_j\>_N={}&\frac{1}{2N}\cos\left(\frac{\pi j}{2N}\right)\left(\cot\left(\frac{\pi(j+k)}{2N}\right)-\cot\left(\frac{\pi(j-k)}{2N}\right)\right).
\end{align*}
Then substitute 
\begin{align*}
\cot(A+B)-\cot(A) = -2\csc(A)\sin(B/2)\cos(B/2)\csc(A+B),
\end{align*}
 giving
\begin{align*}
	\frac{\<\phi_k,\psi_j\>_N}{2N\sin\left(\frac{\pi k}{2N}\right)} = -\frac{\cos\left(\frac{\pi j}{2N}\right)\cos\left(\frac{\pi k}{2N}\right)}{2N^2\sin\left(\frac{\pi (j-k)}{2N}\right)\sin\left(\frac{\pi (j+k)}{2N}\right)}.
\end{align*}
We can conclude that 
\beq
	\frac{b_k b_j\<\phi_k,\psi_j\>_N}{2N\sin\left(\frac{\pi k}{2N}\right)} = -\frac{b_j b_k\<\phi_j,\psi_k\>_N}{2N\sin\left(\frac{\pi j}{2N}\right)} ,
\eeq
and, after canceling pairs for $j>0$, and recalling that $b_0=1$, the double sum reduces to
\begin{align*}
	 \sum_{k=1}^{N-1}\sum_{j=0}^{N-1}\frac{b_k b_j\<\phi_k,\psi_j\>_N}{\lambda_{k,N}} &= \sum_{k=1}^{N-1}\frac{b_k\<\phi_k,\psi_0\>_N}{\lambda_{k,N}}\\
	 	&=\sum_{k=1}^{N-1}\frac{b_k \sigma(k)\cos\left(\frac{\pi k}{2N}\right)}{2N^2\sin^2\left(\frac{\pi k}{2N}\right)},
\end{align*}
where $\sigma(k) = k \mod 2$. We have $2N^2\sin^2\left(\frac{\pi k}{2N}\right)\geq Ck^2$ for $1 \leq k \leq N-1$, so
\begin{align*}
	\sum_{k=1}^{N-1}\sum_{j=0}^{N-1}\frac{b_k b_j\<\phi_k,\psi_j\>_N}{\lambda_{k,N}} \leq C\sum_{k=1}^\infty \frac{1}{k^2}\leq C.
\end{align*}
Finally, noting that $\<\eta,\eta\>_N \leq \<\rho(\eta_\cdot),\eta\>_N+1$, we obtain an estimate for the generator applied to $h_N$, 
\begin{align*}
	\mathcal{L}_Nh_N(\eta)+2\<\eta,\eta\>_N \leq C.
\end{align*}
The process $h_N(X_t)-\int_0^t\mathcal{L}_Nh_N(X_s)ds$ is a martingale, so
\begin{align*}
	E^N\left(h_N(X_T)+2\int_0^T\<X_t,X_t\>_Ndt\right) \leq E^N(h_N(X_0)) + CT, \label{equation:7}
\end{align*}
and since $h_N\geq0$, and $h_N(X_0)\leq\<X_0,X_0\>_N$, 
\begin{align*}
E^N\left[\int_0^T\<X_t,X_t\>_Ndt\right] \leq E^N\left[\<X_0,X_0\>_N\right] + CT,
\end{align*}
where $C$ does not depend on $N$. The proposition below finishes the proof of the lemma.
\end{proof}
\begin{proposition}
	Let $Q^N \rightarrow Q^\infty$ be a weakly convergent sequence of probability measures on $D([0,T],\mathcal{M})$ representing the empirical measures of Markov processes $(X_t,P^N)$ on $\mathbb{N}^{A_N}$. Suppose that
	\beq
		\sup_NE^N\left[\int_0^T\frac{1}{N}\sum_{x\in A_N}X_t(x)^2dt\right] < \infty.
	\eeq
Then $\mu(dx,t) $ is absolutely continuous with density $u \in L^2([0,1]\times[0,T])$ $Q^\infty$-a.s.
\end{proposition}
\begin{proof}
We consider the mollification $K_\epsilon\mu$ defined by 
\beq
K_\epsilon\mu(x,t) = \frac{1}{2\epsilon}\int_{x-\epsilon}^{x+\epsilon}\mu^*(dx,t),
\eeq
 where $\mu^*$ is the projection of $\mu$ onto the torus $\mathbb{T}$ created by identifying $0$ and $1$. We will prove that $\mu^*$ has $L^2$ density, which implies that $\mu$ does. Henceforth we use the notation $\mu$ for simplicity. For given $N$, let $\epsilon = \epsilon_0 + m/N$, where $0\leq\epsilon_0<1/N$.  Then let 
 \beq
 \tilde{K}_\epsilon\mu(x) = \frac{1}{2\epsilon}\int\chi_{(-m/N+\epsilon_0,m/N+\epsilon_0]}d\mu(y). 
 \eeq
 If $\mu$ is the empirical measure for $\eta \in \mathbb{N}^A_N$, then we can calculate $\int_0^1\tilde{K}_\epsilon\mu(x)^2dx$.
 \begin{align*}
	\tilde{K}_\epsilon\mu(x) = \frac{1}{2\epsilon N}\sum_{x\in A_N}\chi_{[x-\epsilon_0,x-\epsilon_0+\frac{1}{N})}\sum_{k=1}^{2m}\eta_{x+\frac{k-m}{N}}.
\end{align*}
So we calculate
\begin{align*}
	\int_0^1\tilde{K}_\epsilon\mu(x)^2dx &= \frac{1}{4\epsilon^2 N^3}\sum_{x\in A_N}\left(\sum_{k=1}^{2m}\eta_{x+\frac{k-m}{N}}\right)^2\\
	&=\frac{1}{4\epsilon^2 N}\left[2m\sum_{x\in A_N}\eta_x^2+\sum_{k=1}^{2m-1}2(2m-k)\sum_{x\in A_N}\eta_x\eta_{x+\frac{k}{N}}\right]\\
	&\leq \frac{1}{4\epsilon^2 N^3}\left(2m+2\sum_{k=1}^{2m-1}k\right)\sum_{x\in A_N}\eta_x^2\\
	&= \frac{m^2}{\epsilon^2N^2}\left(\frac{1}{N}\sum_{x\in A_N}\eta_x^2\right),
\end{align*}
which gives us
\begin{equation*}
	\int_0^1 K_\epsilon\mu(x)^2dx\leq \int_0^1 \tilde{K}_\epsilon\mu(x)^2dx \leq (1+o(1))\frac{1}{N}\sum_{x\in A_N}\eta_x^2.
\eeq
If we can prove that the function $\mu(dx,t) \mapsto  \int_0^T\int_0^1 K_\epsilon\mu(x, t)^2dxdt$ is continuous on $D([0,T],\mathcal{M})$, then $E^N(\int_0^1 K_\epsilon\mu(x)^2dx)$ converges to $E^\infty(\int_0^1 K_\epsilon\mu(x)^2dx)$ for each $N$, and by the above inequality, this quantity is bounded by $\linebreak \sup_N E^N\left[\int_0^T\frac{1}{N}\sum_{x\in A_N}X_t(x)^2dt\right] $, which is finite by hypothesis. Therefore the following two lemmas complete the proof of the proposition.
\end{proof}
\begin{lemma}\label{lemma:ac}
If $\limsup_{\epsilon\rightarrow 0}E^\infty(\int_0^T\int_0^1 K_\epsilon\mu(x)^2dxdt)<\infty$, then with probability one, $\mu$ is absolutely continuous, with $\mu(dx,t) = u(x,t)dx$ and $\linebreak E^\infty(\int_0^T\int_0^1u^2(x)dxdt)<\infty$. 
\end{lemma}
\begin{lemma}
The function $\mu(dx,t) \mapsto  \int_0^T\int_0^1 K_\epsilon\mu(x, t)^2dxdt$ is continuous on $D([0,T],\mathcal{M})$.
\end{lemma}
\begin{proof}
First we prove Lemma \ref{lemma:ac}. Let $f_k(x) = e^{2\pi i kx}$, and $\{f_k\}_{k\in\mathbb{Z}}$ is an orthonormal basis for $L^2([0,1])$, with each $f_k$ continuous on $\mathbb{T}$.  For each measurable function $g:[0,1]\rightarrow\mathbb{R}$, $\int_0^T\int_0^1g(x)^2dxdt = \int_0^T\sum_\mathbb{Z}|\<f_k,g\>|^2dt$. We have $\<f_0, K_\epsilon\mu\> = 1$ for all $\mu \in \mathcal{M}_1$, and for $k\neq0$,
\begin{align*}
\<f_k,K_\epsilon\mu\> &= \int_0^1e^{2\pi ik x}\frac{1}{2\epsilon}\int_{x-\epsilon}^{x+\epsilon}\mu(dy)dx\\
	&= \frac{1}{2\epsilon}\int_0^1\int_{y-\epsilon}^{y+\epsilon}e^{2\pi i k x}dx\mu(dy)\\
	&= \frac{1}{4\pi k i\epsilon}\int_0^1 (e^{2\pi k i\epsilon} - e^{-2\pi k i \epsilon})e^{2\pi k i y}\mu(dy)\\
	&= \frac{\sin(2\pi k \epsilon)}{2 \pi k \epsilon}\<f_k,\mu\>.
\end{align*}
We conclude that as $\epsilon$ goes to zero, $|\<f_k,K_\epsilon\mu\>|$ increases to $|\<f_k,\mu\>|$. Thus by the monotone convergence theorem,
\beq
\lim_{\epsilon\rightarrow0}\int_0^T\sum_{k=-\infty}^\infty|\<f_k,K_\epsilon\mu(t)\>|^2dt = \int_0^T\sum_{k=-\infty}^\infty|\<f_k,\mu(t)\>|^2dt,
\eeq
for all $\mu \in D([0,T],\mathcal{M})$. Let $N = \{\mu \in D([0,T],\mathcal{M}_1):  \int_0^T\sum_{k=\infty}^\infty|\<f_k,\mu\>|^2dt=\infty\}$. Clearly $Q^\infty(N) = 0$ by elementary measure-theoretical considerations. Thus, $ \int_0^T\sum_{k=\infty}^\infty|\<f_k,\mu(t)\>|^2dt<\infty $ $Q^\infty$-a.e. For such $\mu$, the sequence $\<f_k,\mu(t)\>$ is in $l^2$ for almost all $t$, so there is a function $u(x,t) = \sum_{k=-\infty}^\infty\<f_k,\mu(t)\>f_k$ such that $\int_0^T\|u\|_2^2dt<\infty$. To see that $\mu(dx,t) = u(x,t)dx$ as a measure for each $t$ such that $u$ is finite, note that the collection $\{f_k\}$ forms an algebra which is dense in $C(\mathbb{T})$. For each function $g$ in the algebra, $\<g, \mu(t)\> = \<g, u(t)\>$ and the measures must be the same. Now $\mu(dx,t)=u(x,t)dx$ for almost all $t$ and therefore they are equal as functions in $L^2([0,T]\times [0,1])$, as desired.

To prove the continuity lemma, note that as $\mu \rightarrow \nu$ in $\mathcal{M}_1$, $K_\epsilon\mu(x) \rightarrow K_\epsilon\nu(x)$ when the endpoints are points of continuity of $\nu$, that is, almost everywhere. Since $K_\epsilon\mu(x)^2$ is bounded by $\epsilon^{-2}$ for $\mu \in \mathcal{M}_1$, we can apply dominated convergence to see that $\int_0^1 K_\epsilon\mu(x)^2dx \rightarrow  \int_0^1 K_\epsilon\nu(x)^2dx$.  Then if $\mu_\cdot \rightarrow \nu_\cdot$ in $D([0,T],\mathcal{M})$, the time integrals converge, as desired.
\end{proof}

\section{The weak form of the hydrodynamic equation}
In this section we prove that the points of the limiting measure satisfy a weak version of the hydrodynamic equation \ref{eq:stefan1}. First we identify the equation.
\begin{proposition}\label{prop:equivalence}
Let $v$ be a solution to equations \ref{eq:stefan1}-\ref{eq:stefan2}. Then the function $u: [0,1]\times[0,T]\rightarrow \mathbb{R}$ defined by
\beq
u(x,t) = \mathbb{1}_{\{(x,t):x<s(t)\}}(v(x,t)+1)
\eeq
satisfies the integral equation
\begin{multline}\label{eq:weakstefan}
\int_0^1f(x)u(x,t)dx\> - \int_0^1f(x)u(x,0)dx \\
	 =\int_0^t\int_0^1f''(x)\rho(u(x,r))dxdr -a\int_0^t\int_0^1f'(x)u(x,r)dxdr
\end{multline}
for each $t\leq T$.
\end{proposition}
\begin{proof}
This is a calculation by integration by parts. By hypothesis, conditions \ref{eq:stefan1}-\ref{eq:stefan2} hold, and $f'(0) = 0$. Define $v=0$ for $x>s(t)$.
\begin{align*}
	\int_0^{s(t)}f(x)v(x,t)&dx-\int_0^{s(0)}f(x)v(x,0)dx\\
		=&{} \int_0^t\partial_r\int_{0}^{s(r)}f(x)v(x,r)dxdr \\
		=&{}\int_0^t s'(r)f(s(r))v(s(r),r)+\int_0^t\int_0^{s(r)}f(x)\partial_rv(x,r))dxdr\\
		=&{}\int_0^t\int_0^{s(r)}f(x)(v_{xx}(x,r) + av_x(x,r))dxdr\\
		=&{}\int_0^tf(x)(v_x +av)|_{x=0}^{s(r)} - a\int_0^{s(r)}f'(x)v(x,r)dxdr\\
		 &- \int_0^tf'(x)v(x,r)|_{x=0}^{s(r)} + \int_0^{s(r)}f''(x)v(x.r)dxdr\\
\end{align*}
\begin{align*}
		=&{}\int_0^tf(s(r))(-a-s'(r))+af(0) \\
		&+ \int_0^{s(t)}(f''(x) - af'(x))v(x,r)dxdt\\
		=&{}\int_0^t\int_0^{s(t)}-af'(x)dx-d_r\int_0^s(r)f(x)dx\\
		&+\int_0^{s(t)}(f''(x) - af'(x))v(x,r)dxdt\\
		=&{}-\int_0^{s(t)}f(x)dx + \int_0^{s(0)}f(x)dx \\
		&+ \int_0^t\int_0^1f''(x)v(x,t)dx - \int_0^1f'(x)(v(x,t)+1)\mathbb{1}_{x<s(t)}dxdt.
\end{align*}
If it can be shown that $v\geq0$ in $D_T$ for $v_0\geq0$, following our intuition for the heat equation, then the proposition is proven, since $\rho(u)=v$. To see that this holds, consider the space transformation $T(x) = e^{ax}$, and define a function $w$ in $T(D_T)$ by
\begin{align*}
w(T(x),t) = T'(x)v(x,t).
\end{align*}
Then $w$ satisfies $w_t = a^2y^2w_{yy}$ in $T(D_T)$, and by Theorem 2.1 in \cite{Friedman64}, if $w$ has a negative minimum, it occurs on the boundary, which must be at $(1,t_0)$ for some $t_0>0$ because of the other boundary conditions. Then by Theorem 2.14 in the same book, $v_y(1,t_0) \geq 0$. However, the boundary conditions require $v_y(1,t_0) = -a$, a contradiction. Therefore $w\geq0$ on $T(D_T)$, and $v\geq 0$ on $(D_T)$.
\end{proof}

Next we turn to convergence. Our goal is the following lemma, leaving only uniqueness of such solutions to prove Theorem $(0.1)$: 
\begin{lemma}\label{lem:weaksolution}
If, for each $N$, $X^N_0$ under $P^N$ are as in Lemma \ref{lem:l2}, and the corresponding $Q^N$ converge to the delta measure on an absolutely continuous measure with density $u_0 \in L^2$, then under $Q^\infty$, for $0<t\leq T,$ $\mu_t(dx)$ is absolutely continuous a.s.\ with density $u(x,t)$ such that, for $f\in C^2([0,1])$ such that $f'(0)=f'(1)=0,$
\beq
\<f,u(\cdot,t)\>-\<f,u(\cdot,0)\> = \int_0^t\<f'',\rho(u(\cdot,s))\>ds -a\int_0^t\<f',u(\cdot,s)\>ds.
\eeq
\end{lemma}
\begin{proof}
The proof is by elementary estimates. We have already shown that for $f\in C([0,1])$, the $P^N$-martingale $M_t = \<f,X_t\>_N-\int_0^t\mathcal{L}_N\<f,X_s\>_Nds$ satisfies
\beq
P^N\left[|M_t-M_0|>\epsilon\right]<\frac{Ct}{N\epsilon^2},
\eeq
and
\beq
\mathcal{L}_N\<f,X_s\>_N = \<\Delta_Nf,\rho(X_s)\>_N -a\<D_Nf,X_s\>_N,
\eeq
at (\ref{equation:martingalebound}) and (\ref{equation:generator}), respectively. For $f\in C^2([0,1])$ such that $f'(0)=f'(1)=0$, $\Delta_Nf\rightarrow f''$ and $D_Nf\rightarrow f'$ uniformly, so
\beq
\left|\int_0^t\<\Delta_Nf,\rho(X_s)\>_N -a\<D_Nf,X_s\>_Nds - \int_0^t\<f'', \rho(X_s)\>_N-a\<f',X_s\>_Nds\right|
\eeq
goes to zero uniformly in $D([0,T],\mathcal{M}_1)$. Now $-a\<f',X_s\>_N=-a\<f',\mu_s\>$ in distribution, under $P^N$ and $Q^N$ respectively, the right hand side being a continuous functional on $\mathcal{M}_1$. However, the term with $\rho$ is not one, so we must make some additional estimates. For $\mu \in \mathcal{M}$, let 
\beq
M_\epsilon\mu(x) = \frac{1}{2\epsilon}\int\mathbb{1}_{|x-y|\leq\epsilon}d\mu(y).
\eeq
Because of the nature of the proof, this time we extend to the real line rather than the torus when integrating over $y$ beyond the boundary points. In addition, extend $f$ beyond $[0,1]$ so that $f''$ is continuous and bounded, and interpret $\<f'',\rho(M_\epsilon\mu_\eta)\>$ to mean the integral over the whole real line (in this case, just $[-\epsilon, 1+\epsilon]$ because of the support of $M_\epsilon\mu$).  For the process $X$, recall that $S_t = \min\{x\in A_N: X_t = 0\}$, defined to be $(2N+1)/2N$ for the state where all sites have one particle. Under $P^N$, $X_t \in \Omega_N$ for all $t$ a.s., so for $x < S_t$,  $X_t(x) \geq 1$. Let $R_t = S_t-1/N$, the last site with a particle. We claim that for $0<t\leq T$, given $\delta>0$, for $\epsilon$ small enough, as $N$ goes to infinity, 
\beq
Q^N\left[\left|\<f,\mu_t\>-\<f.\mu_0\> - \left(\int_0^t\<f'',\rho(M_\epsilon\mu_s)\>ds -a\int_0^t\<f',\mu_s\>ds\right)\right|>\delta\right] \rightarrow 0.
\eeq
Because of the martingale bound and uniform convergence above, it suffices to show that for small enough $\epsilon$,
\beq
|\<f'',\rho(M_\epsilon\mu_\eta)\>-\<f'',\mu_{\rho(\eta_\cdot)}\>|\label{equation:6}
\eeq
vanishes as $N\rightarrow \infty$. Define
\beq
h(x) = M_\epsilon\mu_\eta(x)-\rho(M_\epsilon\mu_\eta(x)),
\eeq
and note that $0\leq h(x)\leq 1$ and $h(x) = 0$ for $x>R_t+\epsilon$. When $R_t>2\epsilon$, for $x \in [\epsilon,R_t-\epsilon)$, $\eta_y \geq 1$ for $y \in A_N$ such that $|y-x|\leq\epsilon$. So for such $x$, 
\begin{align}\label{equation:hsmall}
\nonumber M_\epsilon\mu_\eta(x) =&{} \frac{1}{2\epsilon}\int\mathbb{1}_{|x-y|\leq\epsilon}d\mu_\eta(y)\\
 \nonumber =&{} \frac{1}{2\epsilon N}\sum_{y\in A_N\cap\{x:|x-y|\leq\epsilon\}}\eta_y\\
 \nonumber \geq &{} \frac{|A_N\cap\{x:|x-y|\leq\epsilon\}|}{2\epsilon N}\\
 \geq &{} \frac{2\epsilon N - 1}{2\epsilon N},
\end{align}
and on this interval, $e(x) \mathrel{\mathop:}= 1-h(x) \leq 1/(2\epsilon N)$.  We will use the decompositions
\begin{align*}
\int f''(x)\rho(M_\epsilon\mu_\eta(x))dx =&{} \int f''(x)M_\epsilon\mu_\eta(x)dx- \int_0^{R_t}f''(x)dx\\ 
	&+ \int_{\epsilon}^{R_t-\epsilon}f''(x)e(x)dx +I_\epsilon,
\end{align*}
where $I_\epsilon$ represents the left over bits of the integrals of bounded functions. Also, we have
\beq
\int f''(x)d\mu_{\rho(\eta_\cdot)}(x) = \int f''(x)d\mu_\eta (x) - \frac{1}{N}\sum_{x\in A_N, x\leq R_t} f''(x).
\eeq
Let $M$ be such that $|f''|<M$, then by (\ref{equation:hsmall}), between $\epsilon$ and $R_t-\epsilon$, $e(t)\leq 1/2\epsilon N$. Therefore we have
\begin{align*}
 \left|\int_{\epsilon}^{R_t-\epsilon}f''(x)e(x)dx\right|\leq\frac{M}{2\epsilon N}.
\end{align*}	
As a Riemann sum, 
\begin{align*}
\left|\int_0^{R_t}f''(x)dx-\frac{1}{N}\sum_{x\in A_N, x\leq R_t} f''(x)\right|\leq C_N,
\end{align*}
for some $C_N$ going to zero, and $|I_\epsilon|\leq 4M\epsilon$. Finally, in the space of distributions, $M_\epsilon\mu \rightarrow \mu$, which gives us convergence of $\<f'',M_\epsilon\mu\>$ to $\<f'',\mu\>$ uniformly on $\mathcal{M}_1$, so we bound the difference by $C_\epsilon\rightarrow 0$. Combining estimates, we get
\begin{align*}
|\<f'',\rho(M_\epsilon\mu_\eta)\>-\<f'',\mu_{\rho(\eta_\cdot)}\>|\leq\frac{M}{2\epsilon N}+ 4M\epsilon + C_N + C_\epsilon.
\end{align*}
Given $\delta>0$, for all $\epsilon$ small enough, as $N$ goes to infinity, the right hand side is less than $\delta/t$, and after integrating over $t$ and taking expectations, the claim is proved. 

Next, $\int_0^t\<f'',\rho(M_\epsilon\mu_s)\>ds$ is a continuous functional on $D([0,T]\mathcal{M})$. The proof is the same as for $K_\epsilon$ in Lemma 2.3. Therefore, 
\begin{align*}
Q^\infty\left[\left|\<f,\mu_t\> - \<f,\mu_0\> - \left(\int_0^t\<f'',\rho(M_\epsilon\mu_s)\>ds -a\int_0^t\<f',\mu_s\>ds\right)\right|>\delta\right]=0,
\end{align*}
for small $\epsilon$. Finally, $\mu \in L^2([0,1]\times [0,T])$ $Q^\infty$-a.s., and for such $\mu$, $M_\epsilon\mu \rightarrow \mu$ in $L^2$, so
\begin{align*}
\left|\int_0^t\<f'',\rho(M_\epsilon\mu_s)\>ds-\int_0^t\<f'',\rho(\mu_s)\>ds\right|\rightarrow 0
\end{align*}
a.s. Convergence is uniform on $\mathcal{M}_1\cap L^2$, so 
\begin{align*}
\left|\<f,\mu_t\> - \<f,\mu_0\> - \left(\int_0^t\<f'',\rho(\mu_s)\>ds -a\int_0^t\<f',\mu_s\>ds\right)\right|
\end{align*}
is less than $\delta$ a.s.\ for all $\delta>0$, and the lemma is proved.
\end{proof}

\section{Uniqueness of weak solutions}
In the previous section, we found that a subsequential limit $Q^\infty$ is concentrated on solutions to the equation (\ref{eq:weakstefan}).
We need a uniqueness theorem for such $u$.
\begin{lemma}\label{lem:weakuniqueness}
Given functions $u_0\in L^2([0,1])$, $a\in C([0,T]),$ a function $u \in L^2([0,1] \times [0,T])$ that satisfies the integral equation~(\ref{eq:weakstefan}) for all $f \in C^2([0,1])$ with $f'(0)=f'(1)=0$ is unique. 
\end{lemma}
\begin{proof}
The proof is based on the method in A2.4 of \cite{Kipnis99}, using techniques analagous to those used for the $L^2$ bound \ref{lem:l2}. Suppose $u$ and $u'$ are two solutions for a given $u_0$, and let $\ol{u}_t(x) = u(x,t)-u'(x,t)$ and $\ol{\rho}_t(x) = \rho(u(x,t)) - \rho(u'(x,t))$. Then 
\begin{equation}\label{equation:4}
\partial_t\<f,\ol{u}_t\>=\<f'',\ol{\rho}_t\>-a\<f',\ol{u}_t\>.
\end{equation}
Again let $\psi_0(x)=1$, $\psi_k(x) = \sqrt{2}\cos(\pi kx)$ for $k>0$, and $\phi_k(x)=\sqrt{2}\sin(\pi kx)$. Then $\{\psi_k\}_{k=0}^\infty$ is an orthonormal basis for $L^2([0,1])$ (the eigenfunctions for the Neumann problem). Note that $\< \psi_0,\ol{u}_t\>=0$ for all $t$, and let $b_k(t)=\<\psi_k,\ol{u}_t\>$. For positive integer $N$, define
\beq
R_N(t) = \sum_{k=1}^N\frac{b_k^2(t)}{k^2},
\eeq
a positive, differentiable function with
\beq
\partial_t R_N(t) = -2\pi^2\sum_{k=1}^N b_k(t)\<\psi_k,\ol{\rho}_t\>-2\pi a(t)\sum_{k=1}^N \frac{b_k\<\phi_k,\ol{u}_t\>}{k}.
\eeq
Expanding $\ol{u}_t$ in terms of $\{\psi_k\}$, we get
\beq
\sum_{k=1}^N \frac{b_k\<\phi_k,\ol{u}_t\>}{k}=\sum_{k=1}^N \sum_{j=1}^\infty \frac{b_kb_j\<\phi_k,\psi_j\>}{k}.
\eeq
Define $\sigma(k,j)$ to be $1$ when $k-j$ is odd and $0$ otherwise, then
\beq
\<\phi_k,\psi_j\> = \frac{2k\sigma(k,j)}{\pi(k^2-j^2)},
\eeq
defined to be $0$ for $k=j$. Now we have
\beq
R_N(t) = \int_0^t-2\pi^2\sum_{k=1}^N b_k(s)\<\psi_k, \ol{\rho}_s\>-4a\sum_{k=1}^N\sum_{j=1}^\infty\frac{b_k(s)b_j(s)\sigma(k,j)}{k^2-j^2}ds.
\eeq
Absolute convergence of these sums to an $L^1([0,T])$ function will allow us to use dominated convergence. By Schwarz's inequality,
\beq
\sum_{k=1}^\infty |b_k(s)\<\psi_k, \ol{\rho}_s\>|\leq\|\ol{u}_t\|_{L^2([0,1])}\|\ol{\rho}_t\|_{L^2([0,1])} \leq \|\ol{u}_t\|_{L^2([0,1])}^2,
\eeq
which is in $L^1([0,T])$ by hypothesis. Observe that $(k^2-j^2)^2 \geq k^2(k-j)^2$, and for all $k$,
\beq
\sum_{j=1}^\infty\frac{\sigma(k,j)}{(j-k)^2}\leq\sum_{m=1}^\infty\frac{2}{m^2}.
\eeq
Then,
\begin{align*}
\sum_{k=1}^\infty\sum_{j=1}^\infty\left|\frac{b_k(s)b_j(s)\sigma(k,j)}{k^2-j^2}\right| &=\sum_{k=1}^\infty|b_k|\sum_{j=0}^\infty\left|\frac{b_j\sigma(k,j)}{k^2-j^2}	\right|\\
		&\leq \sum_{k=1}^\infty|b_k|\left(\sum_{j=0}^\infty b_j^2\right)^{1/2}\left(\sum_{j=0}^\infty\frac{\sigma(k,j)}{(k^2-j^2)^2}\right)^{1/2}\\
		&\leq \sum_{k=1}^\infty|b_k|\left(\sum_{j=0}^\infty b_j^2\right)^{1/2}\left(\frac{1}{k^2}\sum_{m=1}^\infty\frac{2}{m^2}\right)^{1/2}\\
		&= C\|\ol{u}_t\|_{L^2([0,1])}\sum_{k=1}^\infty\frac{b_k}{k}\\
		&\leq C\|\ol{u}_t\|_{L^2([0,1])}^2.
\end{align*}

We apply dominated convergence to conclude that
\begin{multline}
R(t) = \int_0^t-2\pi^2\sum_{k=1}^\infty b_k(s)\<\psi_k, \ol{\rho}_s\>-4a(s)\sum_{k=1}^\infty\sum_{j=1}^\infty\frac{b_k(s)b_j(s)\sigma(k,j)}{k^2-j^2}ds \\=\int_0^t-2\pi^2\sum_{k=1}^\infty b_k(s)\<\psi_k, \ol{\rho}_s\>ds.
\end{multline}
Now,
\beq
\sum_{k=1}^\infty b_k(s)\<\psi_k, \ol{\rho}_s\> = \< \ol{u}_t,\ol{\rho}_t \> \geq 0,
\eeq
since $\rho$ is increasing.  Since $R(0)=0$ by hypothesis, $R(t) \leq 0$, so $R(t) = 0$ for all $t\geq 0$, each $b_k(t)=0$, and $\ol{u}_t=0$, as desired. 
\end{proof}
We can now prove Theorem \ref{thm:1}. 
\begin{theorem}\label{thm:2}
Suppose that for each $P^N$, $X^N_0$ is a random vector with values a.s.\ in $\Omega_N^*$ such that $\sup_NE^N[1/N\sum_{x\in A_N}X^N_0(x)^2]<\infty$, and that its empirical measures $\{\mu^N_0(dx)\}_N$ converge weakly to the delta measure on a fixed absolutely continuous measure $u_0(x)dx$. Then a limiting measure $Q^\infty$ of the corresponding measures $Q^N$ on $D([0,T],\mathcal{M})$ exists and is the delta measure on the unique solution of a weak version (\ref{eq:weakstefan}) of the Stefan problem (\ref{eq:stefan1})-(\ref{eq:stefan2}) with initial data $u_0$. If a solution for the problem (\ref{eq:stefan1})-(\ref{eq:stefan2}) with initial data $(v_0,s_0) = (\rho(u_0),\inf\{x:u_0(x)=0\})$ exists, then the empirical measures of the process $Y_t$ converges to that solution in probability.
\end{theorem}
\begin{proof}
The first part of the theorem follows from Lemmas \ref{lem:relativecompactness}, \ref{lem:l2}, \ref{lem:weaksolution}, and \ref{lem:weakuniqueness}. Given a strong solution $(v,s)$ with initial data $(\rho(u_0),\inf\{x:u_0(x)=0\})$, we have $u_0(x) = \rho(u_0(x)) + \mathbb{1}_{[0,s_0]}$, which is easily checked since $X_t \in \Omega_N^*$ a.s. Therefore, by Proposition \ref{prop:equivalence}, $u(x) \mathrel{\mathop:}= v(x) + \mathbb{1}_{[0,s(t)]}$ is a solution to (\ref{eq:weakstefan}) and, by uniqueness, is the density of the limiting measure of the process $X_t$. In order for $Y_t = \rho(X_T)$ to converge to $v = \rho(u)$, we need $\mu_{\rho(X_t)} \rightarrow \rho(u(x,t))dx$. Indeed, this fact is proved in the proof of Lemma 3.2, and the second part of the theorem is proved.
\end{proof}

\section{Elastic exclusion}
Finally, we return to the regime of reflecting intervals by describing the hydrodynamic limit of the exclusion process $Z_t$. In this section we presume existence of a strong solution, a reasonable assumption that makes the following proofs and calculations more natural. First we describe the dynamics of $Z_t$ more precisely. 

Each site is occupied by zero or one particles, except for $1-1/2N$, which is always empty. Given $Z_t = \theta$, a state $\zeta$ is accessible in one of three cases: first, if for some $1\leq k<N-1$, $\theta_{(2k-1)/2N}=1, \theta_{(2k+1)/2N} = 0$, $\zeta_{(2k-1)/2N}=0, \zeta_{(2k+1)/2N} = 1$, and otherwise the states are equal. Second, if for $k=N-1$, we instead have $\zeta_{(2k+1)/2N}=0$, so that the particle is "killed" here. In either of these case, $\theta$ moves to $\zeta$ with rate $N^2$ times the length of the block of occupied sites to the left of $(2k+1)/2N$ in $\theta$. Third, if $x$ is the first unoccupied site of $\theta$, the state $\zeta$ such that $\zeta_x=1$ and $\zeta_y=\theta_y$ elsewhere is reached at rate $aN$, and we think of the entire block being pushed over to make room for a new particle.
For each $N$, define, for $x\in A_N, x<S_t$,
\beq
U_t(x) = x + \frac{1}{N}\sum_{z \leq x} \rho(X_t(z)).
\eeq
Note that $U_t(S_t-1/N) = 1 - 1/2N.$ Define
\beq
\Psi(X_t)(y) =
\begin{cases}
0 \text{ if } y \in U_t(A_N\cap [0,S_t))\\
1 \text{ else}
\end{cases}
\eeq
so that, for $x<S_t$, between unoccupied sites $y=U_t(x-1/N)$ (or $y=0$ if $x = 1/2N$) and $y'=U_t(x)$ of $\Psi(X_t)$, there are $\rho(X_t(x))$ occupied sites. It is not hard to check that the dynamics of $\Psi(X_t)$ are identical to those of $Z_t$. We can write the inverse map explicitly: For $y \in A_N$ such that $Z_t(y) = 0$,
\beq
T_t (y) \equiv U_t^{-1}(y) = y - \frac{1}{N} \sum_{z<y}Z_t(z),
\eeq
and, given $Z_t$, we extend the domain of $T_t$ to all of $A_N$ with the same formula.  We also define analagous transformations for the solution $v$ of (\ref{eq:stefan1})-(\ref{eq:stefan2}). Let
\beq
\upsilon_t(x) = x + \int_0^x v(x,t)dx,
\eeq
for $0\leq x\leq s(t)$, and, since $v$ is nonnegative and differentiable on $[0,s(t))$, we can define a differentiable inverse $\tau_t = \upsilon^{-1}_t$, and let 
\beq
z(y,t) \mathrel{\mathop:}= \psi(v)(y,t) \mathrel{\mathop:}= v(\tau_t(y),t)\tau'_t(y),
\eeq
so that 
\beq
\tau_t(y) = y - \int_0^y z(y,t)dy.
\eeq
and
\beq
v(x,t) = z(\upsilon_t(x), t)\upsilon_t'(x).
\eeq
One can check that for $z$ so defined, $0\leq z<1$ on $[0,1]$. Given $z:[0,1]\rightarrow [0,1)$, we can similarly define $\tau$ and $v$ with the last two equations, where $\upsilon$ is defined to be the inverse to $\tau$. Next we identify the hydrodynamic equation that functions $z$ should satisfy:
\begin{lemma}
Given functions $z$ and $v$, defined on the appropriate regions, satisfying the relations above for all $t$, $z$ is a solution to the differential equation:
\begin{align}
&\partial_tz(y,t) = \partial_{x}(K(z(y,t))\partial_yz(y,t)) \quad &0<y<1, t>0, \label{equation:9}\\
	&z(x,0) = z_0(x) &0\leq x\leq 1\\
	&K(z(y,t))\partial_yz(y,t)|_{y=0} = a &t>0,\\
	&z(0,t)=0 &t>0\label{equation:8},
\end{align}
with $K(z) = 1/(1-z)^2$, if and only if $v$ is a solution to \ref{eq:stefan1}-\ref{eq:stefan2} with $s(t) = 1-\int_0^1 z(x,t)dx$.
\end{lemma}
\begin{proof}
We will expand the equation
\beq
\partial_{xx}v(x,t) + a\partial_x v(x,t) = \partial_t v(x,t),
\eeq
using
\beq
\partial_x\upsilon_t(x) =1+v(x,t) = \frac{1}{1-z(\upsilon_t(x),t)}.
\eeq
Let $z = z(\upsilon_t(x),t)$, let $z_t$ and $z_y$ refer to the partial derivatives of $z$ with respect to its variables evaluated at $(\upsilon_t(x),t)$. We first check the equivalence of the conditions at the left boundary. The function $v$ is continuous up to the boundary so $\upsilon$ has a derivative at $0$ and the above equation holds.  This gives us equivalence of
\begin{align*}
v_x + a(v+1) &= 0, \\
\frac{-z_y}{(1-z)^3} + \frac{a}{1-z}& = 0,\\
\frac{-z_y}{(1-z)^2} &= -a,
\end{align*}
at $x=\upsilon_t(x)=0$. Next, for $0<x<s(t)$, $0<\upsilon<1$,

\begin{align*}
v_x(x,t) &= \frac{-z_y}{(1-z)^3}, \\
v_{xx}(x,t) &= \partial_x \left(\frac{z_y}{(1-z)^2}\cdot \frac{-1}{1-z}\right)\\
	&= \frac{-1}{1-z}\partial_x\frac{z_y}{(1-z)^2} + \frac{z_y^2}{(1-z)^5},\\
\end{align*}
so
\begin{align*}
v_{xx}+av_x &= \frac{-1}{(1-z)^2}\partial_{\upsilon_t(x)}\frac{z_y}{(1-z)^2}+ \frac{z_y^2}{(1-z)^5} -\frac{az_y}{(1-z)^3}.
\end{align*}
Then
\begin{align*}
\partial_t \upsilon_t(x) &= -\int_0^x \partial_t v(r,t) dr\\
 &= v_x (x,t) + av(x,t) - (v_x(0,t) + av(0,t)).
 \end{align*}
Under either set of conditions, the last term is equal to $-a$, giving
 \begin{align*}
 \partial_t \upsilon_t(x)= \frac{-z_y}{(1-z)^3} + \frac{a}{1-z},
 \end{align*}
 so
 \begin{align*}
 \partial_t v(x,t) &= \frac{-\partial_t \upsilon_t(x)z_y-z_t}{(1-z)^2}\\
 &=\frac{-z_t}{(1-z)^2} + \frac{z_y^2}{(1-z)^5} - \frac{az_y}{(1-z)^3},
 \end{align*}
 and $v_{xx} + av_x = v_t$ is equivalent, for $y = \upsilon_t(x)$, to
 \begin{align*}
 \partial_t(z(y,t)) = \partial_y \frac{\partial_yz(y,t)}{(1-z(y,t))^2},
 \end{align*}
 as desired.
 \end{proof}
 \begin{theorem}\label{thm:5}
 The process $Z_t = \Psi(X_t)$, where $X_t$ has initial data $u_0$ as in Theorem \ref{thm:2}, converges weakly in $D([0,T],\mathcal{M})$ to the measure with density that is the unique solution to the differential equation (\ref{equation:9})-(\ref{equation:8}) with initial data $z_0 = \psi(u_0)$, assuming such a solution exists.
 \end{theorem}
 \begin{proof}
We will prove that for $\delta >0$, $f \in C([0,1])$,
\begin{align*}
P^N\left[ \sup_{0\leq t \leq T} \left| \int_0^1 f(y)\mu_{Z_t(y)}dy - \int_0^1 f(y)z(y,t)dy\right| > \delta \right] \rightarrow 0,
\end{align*}
which suffices to prove the theorem. First, since $v$ is sufficiently smooth, if we extend $f(\upsilon_t(x))$ to be $0$ where it is not defined, then our previous theorem gives
\begin{align*}
P^N\left[ \sup_{0\leq t \leq T} \left| \int_0^1 f(\upsilon_t(x))\mu_{\rho(X)_t(x)}dx - \int_0^1 f(\upsilon_t(x))v(x,t)dx\right| > \delta \right] \rightarrow 0,
\end{align*}
and since the last terms in the two differences are equivalent by a change of variables, we need to look at the difference of 
\begin{align*}
\int_0^1 f(y)\mu_{Z_t(y)} = \frac{1}{N} \sum_{y\in A_N} f(y)Z_t(y)
\end{align*}
and
\begin{align*}
\int_0^1 f(\upsilon_t(x))\mu_{\rho(X_t)}(dx) = \frac{1}{N} \sum_{x \in A_N} f(\upsilon_t(x)) \rho(X_t(x)) = \frac{1}{N} \sum_{y \in A_N} f(\upsilon_t(T_t(y)) Z_t(y).
\end{align*}
Since $Z\leq 1$ and $f$ is continuous on a compact interval, the problem is reduced to the difference of $y$ and $\upsilon_t(T_t(y))$. By Theorem \ref{thm:2} and the fact that the limiting measure is continuous, we have 
\begin{align*}
P^N\left [ \sup_{0\leq t\leq T}\sup_{0\leq x \leq 1} \left|\int_0^x v(z,t)dz - \mu_{\rho(X_t)}(dz)\right|>\delta\right]=P^N\left[E\right] \rightarrow 0,
\end{align*}
where $E$ is the set on the left hand side. For fixed $t$, suppose $y = \upsilon_t(x) = x + \int_0^x v(z,t)dz$, and we claim that for large enough $N$, in the set $E^c$,
\begin{align*}
\left| T_t(y) - x \right | < \delta.
\end{align*}
 Indeed, suppose $T_t(y) < x - \delta$. Then
 \begin{align*}
 y - \frac{1}{N}\sum_{z<y}Z_t(z)  &< x - \delta\\
 \frac{1}{N}\sum_{z \leq T_t(y)} \rho(X_t) & > \int_0^x v(z,t)dz + \delta \\
 \frac{1}{N}\sum_{z \leq T_t(y)} \rho(X_t) & > \frac{1}{N}\sum_{z \leq x} \rho(X_t)(z),
 \end{align*}
 and $T_t(y) > x$, a contradiction. A contradiction results from the opposite inequality in the same way. Then, $\upsilon_t$ being continuous, we obtain a bound on the difference
 \begin{align*}
 \left|\upsilon_t(T_t(y)) - \upsilon_t(x)\right| = \left|\upsilon_t(T_t(y)) - y\right|,
\end{align*}
which gives the result we need and completes the proof of the theorem.
 
\end{proof} 

{\bf Acknowledgements:} The author would like to thank his advisor, Krzysztof Burdzy, for suggesting the problem and for helpful discussions and advice. This research was partially supported by RTG grant 0838212.

\bibliographystyle{amsplain}
\bibliography{jb-bibliography}
\end{document}